\documentclass[11pt, oneside]{amsart}       % use "amsart" instead of "article" for AMSLaTeX format

\usepackage[T1]{fontenc}
\usepackage{mathpazo}
\usepackage{tabularx}
  
\usepackage{hyperref}
\usepackage{booktabs}
\usepackage{xcolor}
\usepackage{tcolorbox}
\hypersetup{
    colorlinks,
    linkcolor={red!50!black},
    citecolor={blue!50!black},
    urlcolor={blue!80!black}
}

\usepackage[all,cmtip]{xy}
\usepackage{graphicx}    
\usepackage{tikz-cd}            % Use pdf, png, jpg, or epsÂ§ with pdflatex; use eps in DVI mode
\usepackage{amsthm,amssymb, amsfonts}   
\usepackage{amsaddr} 
\usepackage{mathrsfs}
\newtheorem{theorem}{Theorem}[section]

\newtheorem{corollary}[theorem]{Corollary}
\newtheorem{definition}[theorem]{Definition}
\newtheorem{lemma}[theorem]{Lemma}
\theoremstyle{remark}
\newtheorem*{remark}{Remark}
\usepackage{tcolorbox}
\newtheorem*{example}{Example}

\renewcommand{\geq}{\geqslant}

\usepackage[utf8]{inputenc}

\newcommand{\YY}{\mathcal{Y}}

\newcommand{\XX}{\mathcal{X}}

\newcommand{\Mod}{\operatorname{Mod}}

\newcommand{\cC}{\mathscr{C}}

\newcommand{\idt}{\operatorname{id}}

\newcommand{\ani}{\mathrm{Ani}}
\newcommand{\fnt}{\mathrm{Fin}}

\title{A Remark on Static Animations}
\author{Emile Bouaziz}
\address{Academia Sinica}
\begin{document}
\maketitle
\begin{abstract}
We record a general condition guaranteeing that the animation of a small $1$-category remains a $1$-category. The proof is extremely elementary universal algebra. This recovers as a very special case the striking observation of Antieau, \cite{Ant}, that the animation of $\fnt^{\mathrm{op}}$ is a $1$-category. Our result is a fair amount more general, and for example applies to $\fnt^{\mathrm{op}}_{\mathcal{G}}$ with $\mathcal{G}$ a groupoid, as well as beyond this.
\end{abstract}

\section{Introduction} \subsection{Animation} Let $\cC$ be a small category \footnote{Throughout \emph{category} is understood to mean $1$-category} with finite coproducts. Then its \emph{animation} is the $\infty$-category $$\text{An}(\cC):=\mathrm{Fun}^{\times}(\cC^{\mathrm{op}},\mathrm{Ani}),$$ where $\mathrm{Ani}$ is the $\infty$-category of anima and the superscript $\times$ indicates that we are looking only at product-preserving functors. In general $\mathrm{An}$ will turn a category into an $\infty$-category, and indeed this is sort of the point. \subsection{Static animations} In the short note \cite{Ant}, Antieau observed the striking fact that the animation of $\fnt^{\text{op}}$ is in fact \emph{static}, which is to say a category. The proof in \cite{Ant} is rather lovely and so we record a sketch here. We must show that $\mathrm{Set}\to\mathrm{Ani}$ induces an equivalence $$\mathrm{Fun}^{\times}(\cC,\mathrm{Set})\to \mathrm{Fun}^{\times}(\cC,\mathrm{Ani}).$$ Let $\XX:\fnt\to\ani$ denote a product-preserving functor. $\XX(\{1,...,p\})$ admits the structure of an abelian group object in $\ani$ on which $p\simeq 0$. Any set $S$ is a retract of $\{1,...,p\}^N$ for large enough $N$. Choosing distinct primes $p$ and $q$ and using an Eckmann-Hilton argument we deduce that all higher homotopy groups of $\XX(S)$ vanish.  The same result appears in \cite{Leh}, see Theorem 5.6 of \emph{loc. cit}. Lehner's proof differs from Antieau's, although it is similarly elementary and elegant. If $\YY\in\ani$ admits an unital associative monoid structure $m$ so that $m$ is idempotent in the sense that (symbolically) $m(y,y)=y$, then $\YY$ is equivalent to a set. Indeed, by Eckmann-Hilton it suffices to note that any group whose multiplication is idempotent is trivial, which is clear. Now any set $S$ admits an idempotent multiplication, for example we can linearly order $S$ and take $m(s,t)=\text{max}\{s,t\}.$ \footnote{We have slightly altered the argument of \cite{Leh} here, but the idea is the same.}

\section{The main result} Our goal in this note is to record a more general result, and supply what is arguably an even simpler proof (in that it does not even invoke Eckmann-Hilton).

Recall, see for example \cite{Cock}, that we call a category $\cC$ \emph{distributive} if it admits finite products and finite coproducts, and products distribute over coproducts.\begin{definition} We will call a category $\cC$ with finite coproducts and finite products \emph{decidable} if for all $X$ there is an object $X^{\neq}$, and a morphism $j_{\neq}:X^{\neq}\to X^2$ so that the induced morphism $\Delta\amalg j_{\neq}: X\amalg X^{\neq}\to X^2$ is an equivalence. \end{definition}

\begin{remark} Note that we require no functoriality in the object $X^{\neq}$; we ask only for its existence. \end{remark}

\begin{example} \begin{itemize}\item  $\fnt$ is decidable; indeed we just take $X^{\neq}$ to be the complement of the diagonal. It is also of course distributive. \item If $G$ is a group then $\fnt_G:=\mathrm{Fun}(BG,\fnt)$ is decidable, again we just take $X^{\neq}$ to be the complement of the diagonal. This works as for all $g$ in $G$ we have $gx=gy\iff x=y$. Again, $\fnt_G$ is obviously distributive. \item The previous example can be generalized to $\mathrm{Fun}(\mathcal{G},\fnt)$, where $\mathcal{G}$ is a groupoid. \item If $M$ is a monoid then the same construction need not work, indeed $X^{\neq}$ will in general not be preserved by the $M$-action, as it may collapse distinct points. \item If $A$ is an algebra then $\Mod(A)$ is decidable as for any module $N$ the diagonal $N\to N\times N\simeq N\oplus N$ is split. On the other hand, it is not distributive. \end{itemize}\end{example}

\begin{definition} Let $\cC$ be a category with finite products and let $x\in\cC$ be an object. We call a morphism $\mu:x^3\to x$ a \emph{majority operation} if we have  $$\mu(a,a,b)=\mu(a,b,a)=\mu(b,a,a)=a.\footnote{The author asked Google's Gemini whether such operations had a name and was informed that \emph{majority} is the established descriptor in the universal algebra literature.}$$ \end{definition}

\begin{remark} We have abused notation above by referring to elements $a,b$ of the object $x\in\cC$. With better manners we would write the identity $\mu(a,a,b)=a$ as $$\mu\circ (\Delta^1_{1,2}\times\idt)=\pi_1,$$ where $\Delta^1_{1,2}:x^2\to x^3$ is the diagonal corresponding to the map $\{1,2,3\}\to \{1,2\}$ given by sending $1,2$ to $1$ and $3$ to $2$.  \end{remark}

\begin{lemma} If $\cC$ is decidable and distributive and if $x \in\cC$, then $x$ admits a majority operation. \end{lemma}

\begin{proof} $x^3=x\times x^2\simeq x\times (x\amalg x^{\neq})$ and by distributivity we can write this as $(x\times x)\amalg x\times x^{\neq}$. This is supplied with an equivalence $(\idt\times\Delta)\sqcup(\idt\times j_{\neq})\to x\times x^2=x^3$ and we define $$\mu=(\pi_2\circ (\idt\times\Delta))\amalg(\pi_1\circ(\idt\times j_{\neq})).$$ This is easily checked to be a majority operation.\end{proof}

\begin{remark} If $x\in \fnt$ then we have defined a morphism $\mu_x:x^3\to x$ by $$\mu_x(a,b,c)=\begin{cases}a,&b\neq c,\\c,&b=c.\end{cases}.$$\end{remark}

\begin{theorem} Let $\cC$ be a distributive decidable category. Then the natural morphism $\mathrm{Fun}^\times(\cC,\mathrm{Set})\to\mathrm{Fun}^\times(\cC,\ani)$ is an equivalence of $\infty$-categories. If $\cC$ is small, then equivalently the animation $\mathrm{An}(\cC^{\mathrm{op}})$ is a $1$-category.\end{theorem}

\begin{proof} As noted in \cite{Ant}, it suffices to prove that if $\XX:\fnt\to\ani$ is product-preserving, then $\XX(x)$ is a set for all $x\in\cC$. Now $\XX(x)\in\ani$ admits a majority operation\footnote{Note that we need not worry about what higher coherences should be inserted as we are going to take homotopy groups. We need only a map $\mu:\XX^3\to \XX$ so that there is \emph{some} path $\mu\circ(\Delta^1_{1,2}\times\idt)\simeq \pi_1$ and similarly for the other majority identities.} which preserves all basepoints as (symbolically) $\mu(x,x,x)=x$ for any majority operation. Then its homotopy groups, $\pi_j(\XX(x),c)\,(j\geq 1)$, also admit majority operations. It suffices then to note that any group $G$ admitting a majority operation (in the category of groups) $\mu:G^3\to G$, is necessarily trivial. Indeed  $$\mu(a,b,c)$$ $$=\mu((a,1,1)*(1,b,1)*(1,1,c))$$ $$=\mu(a,1,1)*\mu(1,b,1)*\mu(1,1,c)$$ $$=1*1*1=1.$$ In particular $\mu(a,a,b)=a=1,$ and we are done. \end{proof}

\begin{remark} After the first version of this note was uploaded to arxiv, Tom De Jong very helpfully pointed out the reference \cite{Taylor}, in which the author proves (among other things) that anima admitting majority operations are sets.\end{remark} 

\begin{corollary} Let $\mathcal{G}$ be a groupoid. Then $\cC=\mathrm{Fun}(\mathcal{G},\fnt)^{\mathrm{op}}$ has a static animation. In particular, this applies to the categories $\fnt^{\mathrm{op}}$ and their $G$-equivariant counterparts. \end{corollary}

\end{document}